\documentclass[12pt,reqno]{amsart}
\usepackage{geometry}
\usepackage[latin1]{inputenc}
\usepackage[italian,english]{babel}
\usepackage{amsmath, amsfonts, amsthm, xcolor}
\usepackage{paralist}
\numberwithin{equation}{section}
\usepackage{mathtools, mathabx, verbatim}

\newtheorem{thm}{Theorem}[section]
\newtheorem{lem}[thm]{Lemma}

\newtheorem{prop}[thm]{Proposition}

\newtheorem{defn}[thm]{Definition}
\theoremstyle{definition}
\newtheorem{rem}[thm]{Remark}
\theoremstyle{remark}

\newtheorem{cntex}[thm]{Counterexample}

\newcommand{\ds}{\displaystyle}

\newcommand{\abs}[1]{\left\vert#1\right\vert}

\newcommand{\R}{\mathbb{R}}
\newcommand{\N}{\mathbb{N}}

\newcommand{\de}{\partial}

\DeclareMathOperator{\diam}{diam}
{\left\{\begin{array}{@{}l@{}}}{\end{array}\right.}
\patchcmd{\abstract}{\scshape\abstractname}{\textbf{\abstractname}}{}{}
\makeatletter 
\def\@makefnmark{} 
\makeatother 

\title[On the first Steklov-Robin eigenvalue ]{ On a Steklov-Robin eigenvalue problem}
\author[N. Gavitone, R. Sannipoli]{
	Nunzia Gavitone, Rossano Sannipoli}
\address{Dipartimento di Matematica e Applicazioni ``R. Caccioppoli'', Universit\`a degli studi di Napoli Federico II \\ Via Cintia, Complesso Universitario Monte S. Angelo, 80126 Napoli, Italy.}
\email{nunzia.gavitone@unina.it}
\address{Dipartimento di Matematica e Applicazioni ``R. Caccioppoli'', Universit\`a degli studi di Napoli Federico II \\ Via Cintia, Complesso Universitario Monte S. Angelo, 80126 Napoli, Italy.}
\email{rossano.sannipoli@unina.it}

\begin{document}
\maketitle
\begin{abstract}
In this paper we study a Steklov-Robin eigenvalue problem for the Laplacian in annular domains. More precisely, we consider $\Omega=\Omega_0 \setminus \overline{B}_{r}$, where $B_{r}$ is the ball centered at the origin with radius $r>0$ and  $\Omega_0\subset\mathbb{R}^n$, $n\geq 2$, is an open, bounded set with Lipschitz boundary, such that $\overline{B}_{r}\subset \Omega_0$. We impose a Steklov condition on the outer boundary and a Robin condition involving a positive $L^{\infty}$ function $\beta(x)$ on the inner boundary. Then, we  study  the  first  eigenvalue $\sigma_{\beta}(\Omega)$  and its main properties. In particular,  we investigate  the behaviour of $\sigma_{\beta}(\Omega)$ when we let vary the $L^1$-norm of $\beta$ and the radius of the inner ball. Furthermore, we study the asymptotic behaviour of the corresponding eigenfunctions when $\beta$ is a positive parameter that  goes to infinity.\\ 

\end{abstract}
\section{Introduction}
The study of eigenvalue problems in annular domains is, actually, a topic widely studied. In this kind of problems, usually,  the shape of the hole is  spherical and  different conditions are imposed on the two components of the boundary. More precisely this problems are defined in annular domains having the following form $\Omega=\Omega_0\setminus \overline{B_{r}}$. Here $\Omega_0\subset\mathbb{R}^n$, $n\geq 2$, is an open, bounded, connected  set with Lipschitz boundary and $B_{r}$ is the  ball of radius $r>0$ centered at the origin such that  $\overline{B_{r}}\subset\Omega_0$. Recently, the following Steklov-Dirichlet Laplacian eigenvalue problem has attracted the attention of many mathematicians  ( see for instance \cite{verma2018bounds,verma2020eigenvalue, gavitone2021isoperimetric,paoli2020stability,hong2020shape, seo2021shape,gavitone2021monotonicity} and the references therein) 

\begin{equation}\label{proSD}
\begin{cases}
\Delta u=0 & \mbox{in}\ \Omega\vspace{0.2cm}\\
\dfrac{\de u}{\de \nu}=\sigma u&\mbox{on}\ \partial\Omega_0\vspace{0.2cm}\\ 
u=0&\mbox{on}\ \partial B_{r}, 
\end{cases}
\end{equation}
where $\nu$ is the outer unit normal to $\partial \Omega$.  The first eigenvalue to the problem \eqref{proSD} has the following variational characterization
\begin{equation} \label{SDir}
     \sigma_D(\Omega)=
\inf_{\substack{v\in H^{1}_{\partial B_r}(\Omega)\\ v\not \equiv0} } \frac{\int_{\Omega}\abs{\nabla v}^2\,dx}{\int_{\partial \Omega_0} u^2\,d\mathcal{H}^{n-1}},
\end{equation}
 where $H^{1}_{\partial B_r}(\Omega)$ is the set of those Sobolev functions vanishing on the boundary of the hole. 
Since both Steklov and Dirichlet  boundary conditions are imposed, it is interesting to understand what properties $\sigma_D(\Omega)$ inherits from them.
 In \cite{paoli2020stability,dittmar2005eigenvalue,hong2020shape}, the authors prove that the infimum is actually a minimum and that it is achieved by a function $u\in H^{1}_{\partial B_r}(\Omega)$ which is a weak solution to problem \eqref{proSD} with constant sign and that 
$\sigma_D(\Omega)$ is simple. This means that, even if a Steklov condition appears in \eqref{proSD},  $\sigma_D(\Omega)$ has the same properties of the first Dirichlet eigenvalue. Nevertheless, in \cite{paoli2020stability,verma2020eigenvalue,hong2020shape,ftouh2022place,seo2021shape,gavitone2021monotonicity} several shape optimization problems related to $\sigma_D(\Omega)$ are studied. In this case it is possible to  observe that the Steklov condition influences the behavior of $\sigma_D(\Omega)$. Indeed, one of the main property of  $\sigma_D(\Omega)$  for instance is that it is bounded from above when both the measure of $\Omega_0$ and the radius of the hole are fixed (see for instance \cite{paoli2020stability,gavitone2021isoperimetric}), unlike the first Dirichlet Laplacian eigenvalue.
 
Our aim is to study the eigenvalue problem obtained replacing the Dirichlet boundary condition in \eqref{proSD} with a Robin one. More precisely we deal  with the following Steklov-Robin eigenvalue problem 
\begin{equation}\label{proSR}
\begin{cases}
\Delta u=0 & \mbox{in}\ \Omega\vspace{0.2cm}\\
\dfrac{\de u}{\de \nu}=\sigma u&\mbox{on}\ \partial\Omega_0\vspace{0.2cm}\\ 
\dfrac{\de u}{\de \nu}+\beta(x) u=0&\mbox{on}\ \partial B_{r}, 
\end{cases}
\end{equation}
where  $\beta(x)\in  L^{\infty}(\partial B_r) $ is a positive function such that $\|\beta\|_{L^1(\partial B_r)}=m>0$.\\

\begin{defn}
	A real number $\sigma(\Omega)$ and a function $u\in H^1(\Omega)$ are, respectively, called eigenvalue of \eqref{proSR} and associated eigenfunction  to $\sigma(\Omega)$, if and only if 
	\begin{equation*}\label{WS}
	\int_{\Omega} \langle\nabla u,\nabla \varphi\rangle \;dx+\int_{\partial B_r}\beta(x)u\varphi \,d\,\mathcal{H}^{n-1}=\sigma(\Omega)\int_{\partial\Omega_0}u \varphi \;d\mathcal{H}^{n-1}
	\end{equation*}
	for every $\varphi\in H^1(\Omega)$.
\end{defn}

The aim of this paper is to study the first eigenvalue $\sigma_{\beta}(\Omega)$ of \eqref{proSR}  defined as (see Section 3 for the details)
\begin{equation}
\label{eig1}
  \sigma_{\beta}(\Omega)  = \inf_{v \in H^1(\Omega)\setminus \{0\}}\dfrac{\ds\int _{\Omega}|\nabla v|^2\;dx+\int_{\partial B_r}\beta(x)v^2\,d\mathcal{H}^{n-1}}{\ds\int_{\partial\Omega_0}v^2\;d\mathcal{H}^{n-1}}.
\end{equation}
We prove that similarly to problem \eqref{proSD},  $\sigma_{\beta}(\Omega)$ is actually a minimum, it is simple, and that the corresponding eigenfunctions have constant sign. Hence, also in this case, despite the Steklov condition, $\sigma_{\beta}(\Omega)$  is formally a Robin type eigenvalue.

When $\Omega$ is a spherical shell, that is $\Omega =B_R\setminus \overline{B_r}$, and $\beta(x)=\beta$ is a positive constant,  $\sigma_{\beta}(\Omega)$ and the corresponding eigenfunctions can be explicitly computed; in particular, in the radial case, it is straightforward to prove the asymptotic behaviors of $\sigma_{\beta}(\Omega)$ and the corresponding eigenfunctions when $\beta\to 0$ and $\beta \to \infty$ (see Section 4 for the details).

We observe that $\sigma_{\beta}(\Omega)$ depends clearly also on $\beta$ and consequently on its $L^1$-norm $m$. When $\Omega$ is a general open, bounded set with Liptschitz boundary, we expect that the same asymptotic behavior of $\sigma_{\beta}(\Omega)$ and the corresponding eigenfunctions still holds, when we let vary the $L^1-$norm $m$ of $\beta(x)$.
In order to show this,   we will prove some estimates on $\sigma_{\beta}(\Omega)$ in the spirit of the ones contained in \cite{kuttler1973note} for the first Robin Laplacian eigenvalue (see also \cite{di2022two} for a more general case). More precisely, let us define the following quantities
\begin{equation}
\label{kutmu}
\mu_1(\Omega):= \inf_{\substack{v \in H^{1}(\Omega)\setminus \{0\}\\ \int_{\partial B_r}v\,d \mathcal H^{n-1}=0}}\displaystyle \frac{\int_{\Omega}|\nabla v|^2\,dx}{\int_{\partial\Omega_0}v^2\,d\mathcal{H}^{n-1}},
\end{equation}
and if $u$ is the eigenfunction corresponding to $\sigma_{\beta}(\Omega)$
\begin{equation}
\label{q1}
    q_\beta(\Omega) = \inf_{\ds\substack{\Delta w=0, \, w\neq 0\\ \frac{\partial w}{\partial \nu}=0 \, \textit{on}\,\, \partial \Omega_0}}\frac{\ds\int_{\partial B_r}\beta(x)w^2\,d\mathcal{H}^{n-1}}{\ds\int_{\partial \Omega_0}w^2\,d\mathcal{H}^{n-1}},
\end{equation}
We observe that $\mu_1(\Omega)$ is the first nontrivial Steklov Laplacian  eigenvalue in $\Omega$. Moreover if $\beta(x)=\beta$ is constant, then
\begin{equation}
    q_\beta(\Omega) = \beta q(\Omega),
\end{equation}
where 
\begin{equation}\label{qomega}
     q(\Omega) = \inf_{\ds\substack{\Delta w=0, \, w\neq 0\\ \frac{\partial w}{\partial \nu}=0 \, \textit{on}\,\, \partial \Omega_0}}\frac{\ds\int_{\partial B_r}w^2\,d\mathcal{H}^{n-1}}{\ds\int_{\partial \Omega_0}w^2\,d\mathcal{H}^{n-1}}.
\end{equation}

Then our result is the following
\begin{thm}\label{mainstime}
     Let $\Omega_0 \subset \R^n$ be an open, bounded set with Lipschitz boundary and let $\Omega=\Omega_0 \setminus \overline{B_r}$, where $B_r$ is the ball centered at the origin and with radius $r$ such that $\overline{B_r}\subset \Omega_0$. Then the following estimates hold
	\label{main}
	\begin{equation}\label{Kutt_mi}
	\frac{1}{\sigma_{\beta}(\Omega)}\leq\frac{1}{\mu_1(\Omega)}+\frac{P(\Omega_0)}{m},\\
	\end{equation}
	and
	\begin{equation}\label{Kutt_40i}
	\frac{1}{\sigma_{\beta}(\Omega)}\leq\frac{1}{\sigma_D(\Omega)}+\frac{1}{ q_\beta(\Omega) },
	\end{equation}
	where $\mu_1(\Omega)$ and $q_\beta(\Omega)$ are defined in \eqref{kutmu} and \eqref{q1}, and $\sigma_{\beta}(\Omega)$ and $\sigma_D(\Omega)$  are  the first Steklov-Robin  and Steklov Dirichlet eigenvalue of $\Omega$  defined in \eqref{eig1} and \eqref{SDir} respectively.
\end{thm}
As a consequence of the above estimates we can obtain the quoted asymptotic behaviour of  $\sigma_{\beta}(\Omega)$ with respect to $m$. 
We stress the fact that, when $\beta(x)=\beta$ is a positive constant, Theorem \ref{mainstime} is a key point to prove the convergence of the Steklov-Robin eigenfunctions to the Steklov-Dirichlet eigenfunction as $\beta$ goes to infinity. Indeed, our second main result is the following
\begin{thm}\label{convergence}
Let $\Omega_0 \subset \R^n$ be an open, bounded set with Lipschitz boundary and let $\Omega=\Omega_0 \setminus \overline{B_r}$, where $B_r$ is the ball centered at the origin and with radius $r$ such that $\overline{B_r}\subset \Omega_0$. Let $\beta(x)=\beta$ be a positive constant. Let $u_\beta$ and $v$ be  positive eigenfunctions corresponding to $\sigma_\beta(\Omega)$ and $\sigma_D(\Omega)$ respectively. Then
\begin{equation} \label{H1conv}
    \lim_{\beta \to +\infty} \|u_\beta -v\|_{H^1(\Omega)}= 0.
\end{equation}
\end{thm}
\begin{rem}
    The previous result states that the first  Steklov-Robin eigenfunctions converge to the Steklov-Dirichlet ones in $H^1(\Omega)$, when $\beta$ goes to $\infty$. Actually, since they are smooth functions, by Theorem \ref{mainstime} and  Morrey's inequality  we can say that the convergence result holds true in  $C^0(\Omega).$
\end{rem}

The structure of the paper is the following: in Section 2 we recall some preliminary result and fix the notation we will use in the rest of the paper. In Section 3 we state and prove the main properties of the first eigenvalue $\sigma_{\beta}(\Omega)$ and of the corresponding eigenfunctions. In Section 4 we focus on the radial case and finally Section 5 and 6 are devoted for the computation of some upper and lower bounds for $\sigma_{\beta}(\Omega)$ and the proof of Theorems \ref{mainstime} - \ref{convergence}.

\section{Preliminaries and notations}
Throughout this paper, we denote by $B_r(x_0)$ and $B_r$ the balls in $\mathbb{R}^n$ of radius $r>0$ centered at  $x_0\in\mathbb{R}^n$ and at the origin, respectively. Moreover with $B$, $\mathbb{S}^{n-1}$ and $\omega_n$ we will denote respectively the unit ball of $\mathbb R^{n}$, its boundary and its volume. Let $r,R$ be such that $0<r<R$, the spherical shell will be denoted as follows:
\begin{equation*}
A_{r,R}=\{x \in \mathbb R^{n} \colon r<|x|<R\}.
\end{equation*}
Moreover, the $(n-1)$-dimensional Hausdorff measure in $\mathbb{R}^n$ will be denoted by $\mathcal H^{n-1}$ and the Euclidean scalar product in $\mathbb{R}^n$ is denoted by $\langle\cdot,\cdot\rangle$.\\
Let $D\subseteq\mathbb{R}^n$ be an open bounded set and let $E\subseteq\R^{n}$ be a measurable set.  For the sake of completeness, we recall here the definition of the perimeter of $E$ in $D$ (see for instance \cite{maggi2012sets}), that is
\begin{equation*}
P(E;D)=\sup\left\{  \int_E {\rm div} \varphi\:dx :\;\varphi\in C^{\infty}_c(D;\mathbb{R}^n),\;||\varphi||_{\infty}\leq 1 \right\}.
\end{equation*}
The perimeter of $E$ in $\mathbb{R}^n$ will be denoted by $P(E)$ and, if $P(E)<\infty$, we say that $E$ is a set of finite perimeter. Moreover, if $E$ has Lipschitz boundary, it holds
\[
P(E)=\mathcal H^{n-1}(\partial E).
\]
The Lebesgue measure of a measurable set $E \subset \mathbb R^{n}$ will be denoted by $V(E)$.
Moreover, we define the inradius of $E\subset \mathbb R^{n}$ as 
\begin{equation}\label{inradius_def}
\rho( E)=\sup_{x\in E}\; \inf_{y\in\partial E} |x-y|,
\end{equation} 
while the diameter of $E$ is
\[
\diam (E)=\sup_{x,y \in E} |x-y|.
\]
\subsection{Some properties of convex sets}
In what follows we recall some properties of the convex bodies, i.e. compact convex sets without empty interior. We refer to  \cite{schneider2014convex} for further properties and the details.

Let $K \subset \R^n$ be a bounded convex body. The support function $h_K$ of $K$ is defined as
follows
\[
h_K \colon \mathbb S^{n-1}\to  \R, \quad  h_K(x) = \sup_{y \in K}(x,y).\]
If the origin belongs to $K$ then $h_K$ is non-negative and $h_K(x) \leq \text{diam}(K)$ for every $ x \in \mathbb S^{n-1}$. 	
Let $K \subset \R^n$ be a bounded convex body such that the origin is an interior point of $K$. The radial function of $K$ is defined as follows
\begin{equation} 
\label{ro}
\rho_K(x)=\sup\{\lambda\ge 0 \colon \lambda x \in K\}, \quad x\in \mathbb S^{n-1},
\end{equation}
and it is a Lipschitz function. 
Moreover the boundary of $K$ can be written as 
\begin{equation}
\label{rad}
\partial K = \{ x\,\rho_K(x),  x \in \mathbb S^{n-1}\}.
\end{equation}
Let us define the minimum and the maximum distance of $\partial K$ from the origin as follows
\begin{equation}
\label{re}
R_m=\min_{\mathbb S^{n-1}}\rho_K(x), \qquad\quad R_M=\max_{\mathbb S^{n-1}}\rho_K(x).
\end{equation}
Finally, if $f\colon \partial K \to \R$ is $\mathcal{H}^{n-1}-$integrable  the following formula for the change of variable given by the radial map holds
\[
\int_{\partial K} f \, d \mathcal H^{n-1}=\int_{\mathbb S^{n-1}}f(r_K(x))(\rho_K(x))^{n-1}\sqrt{1+\left(\frac{|\nabla_\tau \rho_K(x)|}{\rho_K(x)}\right)^2}\,d\mathcal H^{n-1}.
\]
where $\nabla_{\tau}$ denotes the tangential gradient, which is the projection of the gradient on the tangent plane.

\subsection{Trace and Friedrich inequalities}
Here we recall some known inequalities that will be useful in the sequel. Let $\Omega$ be an open bounded subset of $\mathbb{R}^n$ with Lipschitz boundary, then by the classical Sobolev trace inequality (see \cite{evans2010partial}) we have that  \begin{equation}\label{tracein}
    \|u\|_{L^2(\partial \Omega)} \le C \|u\|_{H^1(\Omega)},
\end{equation}
for some positive constant $C>0$. Moreover the embedding operator of $H^1(\Omega)$ into $L^2(\partial \Omega)$ is compact.\\
Another important embedding theorem is a consequence of the so-called  Friedrich's inequality (see for instance \cite{friedrichs1928randwert, maz2013sobolev} and for a more general case \cite{cianchi2016sobolev}). 
Let $H^1(\Omega,\partial \Omega)$ the completion of the set of functions in $C^{\infty}(\Omega)\cap C(\Bar{\Omega})$ which have weak gradient in $L^2(\Omega)$, equipped with the following norm (see \cite{maz2013sobolev} for the details)

\begin{equation*}
    \|u\|_{H^1(\Omega,\partial \Omega)}=\|\nabla u\|_{L^2(\Omega)}+\|u\|_{L^2(\partial\Omega)}.
\end{equation*}
Fridrich's inequality states that
\begin{equation}\label{friedin}
    \|u\|_{L^2(\Omega)}\le C(\|\nabla u\|_{L^2(\Omega)}+\|u\|_{L^2(\partial\Omega)})
\end{equation}
for some positive constant $C>0$. Also in this case the embedding operator of $H^1(\Omega,\partial \Omega)$ into $L^2(\Omega)$ is compact (see Corollary 3, p. 392 in \cite{maz2013sobolev}).
\section{Existence and basic properties of $\sigma_{\beta}(\Omega)$} \label{secEX}

In this section we define and study the main properties of the first Steklov-Robin Laplacian eigenvalue in $\Omega=\Omega_0 \setminus \overline{B}_r$.  Let $\sigma_{\beta}(\Omega)$ be the following quantity
\begin{equation} \label{inf}
\sigma_{\beta}( \Omega)=\inf_{\substack{v\in H^{1}(\Omega)\\ v\not \equiv0} }    J[v],
\end{equation}
where 
\begin{equation}\label{JV}
    J[v]= \dfrac{\ds\int _{\Omega}|\nabla v|^2\;dx+\int_{\partial B_r}\beta(x)v^2\,d\mathcal{H}^{n-1}}{\ds\int_{\partial\Omega_0}v^2\;d\mathcal{H}^{n-1}}
\end{equation}
and $\beta(x)\in L^{\infty}(\partial B_r)$ is a positive function.\\
We observe that by \eqref{inf} we immediately get 
\begin{equation}\label{SBminSD}
	\sigma_{\beta}(\Omega)\le \sigma_D(\Omega),
\end{equation}
where $\sigma_D(\Omega)$ is the first Steklov-Dirichlet eigenvalue defined in \eqref{SDir}.
In the next result we prove that $\sigma_{\beta}(\Omega)$ is the first eigenvalue of problem \eqref{proSR} and we show some basic properties of $\sigma_{\beta}(\Omega)$ and its corresponding eigenfunctions.

\begin{thm}
Let $n\ge 2$ and $\Omega=\Omega_0\setminus \overline{B}_r$, where $\Omega_0$ is an open, bounded and connected set with Lipschitz boundary in $\mathbb{R}^n$ and $B_r$ a ball centered at the origin of radius $r>0$ such that $\overline{B}_r\subset \Omega_0 $. Then $\sigma_{\beta}(\Omega)$ is actually a minimum, that is
\begin{equation}\label{varcar}
	\sigma_{\beta}( \Omega)=\min_{\substack{v\in H^{1}(\Omega)\\ v\not \equiv0} }    J[v],
\end{equation}
where $J[v]$ is defined in \eqref{JV}. Moreover $\sigma_{\beta}(\Omega)$ is the first eigenvalue of \eqref{proSR}, it is strictly positive and any minimizer has constant sign.
\end{thm}
\begin{proof}
	Let us notice that the Rayleigh quotient $J[w]$ defined in the previous proposition is always non-negative and $0$-homogeneous. Let us consider a minimizing normalized sequence $\{u_n\}_{n\in\mathbb{N}}$ such that $\|u_n\|_{L^2(\partial \Omega_0)}=1$, i.e. $\lim_{n\to \infty} J[u_n]= \sigma_{\beta}(\Omega)$. By \eqref{SBminSD},  $J[u_n]\le \sigma_D(\Omega)$, then by Friedrich's inequality \eqref{friedin},  $\|u_n\|_{L^2(\Omega)}$ is uniformly bounded from above. By the compactness of the embedding $H^1(\Omega,\partial\Omega)\subset L^2(\Omega)$, there exists a subsequence, still denoted by $u_n$, and a function $u\in H^1(\Omega)$ with $||u||_{L^2(\partial\Omega_0)}=1$, such that $u_n\to u$ strongly in $L^2(\Omega)$, hence also almost everywhere, and $\nabla u_n\rightharpoonup \nabla u$ weakly in $L^2(\Omega)$. Moreover, by the compactness of the trace embedding theorem \eqref{tracein},
\phantom{ }$u_n$ converges strongly  to $u$ also in $L^2(\de \Omega)$ and almost everywhere on $\de \Omega$ to  $u$. Then, by weak lower semicontinuity we have
\begin{equation*}
\lim\limits_{n\to+\infty }J[u_n]\geq J[u].
\end{equation*}
Hence the existence of a minimizer $u\in H^1(\Omega)$ follows.\\

It is obvious the fact that $\sigma_{\beta}(\Omega)\ge 0$. By contradiction let us suppose that $\sigma_{\beta}(\Omega)=0$. This means that
\begin{equation*}
    \int_{\Omega}|\nabla u|^2\,dx + \int_{\partial B_r}\beta(x)u^2\,d\mathcal{H}^{n-1}=0.
\end{equation*}
It follows that $\|\nabla u\|_{L^2(\Omega)}$ and $\|\sqrt{\beta(x)}u\|_{L^2(\partial B_r)}$ are both zero. From the first we have that $u$ is constant a.e. in $\Omega$ and then it must be $\beta(x)=0$ a.e. on $\partial B_r$, which is an absurd. Therefore $\sigma_{\beta}(\Omega)>0$.\\
By classical arguments of Calculus of Variation it is easy to prove that \eqref{proSR} is the Euler-Lagrange equation corresponding to \eqref{varcar}. Here we write down the proof for completeness. Let $u\in H^1(\Omega)$ be minimum of the Rayleigh quotient \eqref{JV} and let $\lambda\in\mathbb{R}$ its value, i.e. $J[u]=\lambda$. Let us now consider the first variation of $J[\cdot]$. If $v\in H^1(\Omega),$ we define the following function
	\begin{equation*}
	f(\varepsilon)= J[u+\varepsilon v].
	\end{equation*}
	It is clear that $f(0)=\lambda$ and in particular we have that
	\begin{equation*}
	\begin{split}
		\int_{\partial \Omega_0}u^2 \,d\mathcal{H}^{n-1}\cdot f'(0)&=2\bigg(\int_{\Omega}\langle \nabla u,\nabla v\rangle\,dx+ \int_{\partial B_r}\beta(x)uv\,d\mathcal{H}^{n-1}\bigg)\int_{\partial \Omega_0}u^2 \,d\mathcal{H}^{n-1}\\
		&-\bigg(\int_{\Omega}|\nabla u|^2 \,dx+ \int_{\partial B_r}\beta(x)u^2\,d\mathcal{H}^{n-1}\bigg)\int_{\partial \Omega_0}uv \,d\mathcal{H}^{n-1}=0
	\end{split}
	\end{equation*}
if and only if
\begin{equation*}
	\dfrac{\int_{\Omega}\langle \nabla u,\nabla v\rangle\,dx+ \int_{\partial B_r}\beta(x)uv\,d\mathcal{H}^{n-1}}{\int_{\partial \Omega_0}uv \,d\mathcal{H}^{n-1}}= \dfrac{\int_{\Omega}|\nabla u|^2 \,dx+ \int_{\partial B_r}\beta(x)u^2\,d\mathcal{H}^{n-1}}{\int_{\partial \Omega_0}u^2 \,d\mathcal{H}^{n-1}}=\lambda.
\end{equation*}
Since the relation written above is valid for every $v \in H^1(\Omega)$, the proposition is proved by definition of weak solution. \\
In particular it follows that $\sigma_{\beta}(\Omega)$ is the smallest eigenvalue of problem \eqref{proSR}. Indeed let us suppose that $v$ is another eigenfunction of \eqref{proSR} with corresponding eigenvalue $\tilde{\sigma}$. Then an integration by parts gives
\begin{equation*}
\sigma_{\beta}(\Omega) \le \dfrac{\int_{\Omega}|\nabla v|^2 \,dx+ \int_{\partial B_r}\beta(x)v^2\,d\mathcal{H}^{n-1}}{\int_{\partial \Omega_0}v^2 \,d\mathcal{H}^{n-1}} = \dfrac{\int_{\partial \Omega}\frac{\partial u}{\partial \nu}u\,d\mathcal{H}^{n-1}+ \int_{\partial B_r}\beta(x)u^2\,d\mathcal{H}^{n-1}}{\int_{\partial \Omega_0}u^2 \,d\mathcal{H}^{n-1}}=\tilde{\sigma}.
\end{equation*}
It only remains to show that any minimizer has constant sign. If $u$ be an eigenfunction corresponding to $\sigma_{\beta}(\Omega)$, then $J[u]=J[|u|]$. This implies that $u=|u|$ on $\Omega$ and therefore $u\geq0$ on $\Omega$. By Harnack inequality (see \cite[Thm 1.1]{trudinger1967harnack}), $u$ is strictly positive on $\Omega$.\\
\end{proof}

Next propositions concern the simplicity of $\sigma_{\beta}(\Omega)$ and sign properties of the corresponding eigenfunctions.

\begin{prop}
$\sigma_{\beta}(\Omega)$ is simple, which means that there exists a unique corresponding eigenfunction up to multiplicative constants.
\end{prop}
\begin{proof}
Let us now suppose that $v$ is another eigenfunction corresponding to $\sigma_{\beta}(\Omega)$. Since $v>0$ in $\Omega$, it follows
\begin{equation*}
	\int_{\Omega} v \,dx\neq 0.
\end{equation*}
Hence there exists a positive number $\xi >0$, such that
\begin{equation*}
	\int_{\Omega}(u-\xi v)\,dx=0.
\end{equation*}
Since $u-\xi v$ is another eigenfunction corresponding to the same eigenvalue, it is necessary that $u=\xi v$ in $\Omega$, which proves the simplicity.
\end{proof}

\begin{prop}
Let $n\ge 2$ and $\Omega=\Omega_0\setminus \overline{B_r}$, where $\Omega_0$ is an open, bounded and connected set with Lipschitz boundary in $\mathbb{R}^n$ and $B_r$ a ball centered at the origin of radius $r>0$ such that $\overline{B_r}\subset \Omega_0 $.
Any nonnegative function $v\in H^1(\Omega)$ that satisfies in the sense of definition \eqref{WS} 
\begin{equation}\label{EQHARM}
    \begin{cases}
\Delta u=0 & \mbox{in}\ \Omega\vspace{0.2cm}\\
\dfrac{\de u}{\de \nu}=\sigma u&\mbox{on}\ \partial\Omega_0\vspace{0.2cm}\\ 
\dfrac{\de u}{\de \nu}+\beta(x) u=0&\mbox{on}\ \partial B_{R}, 
\end{cases}
\end{equation}
is a first eigenfunction of \eqref{EQHARM}, that is $\sigma = \sigma_{\beta}(\Omega)$, and $v=u$ (up to multiplicative constants), where $u$ is the eigenfunction corresponding to the first eigenvalue $\sigma_{\beta}(\Omega)$. 
\end{prop}
\begin{proof}
Since $u$ is a positive eigenfunction corresponding to $\sigma_{\beta}(\Omega)$, it satisfies
\begin{equation}\label{testfuncu}
	\int_{\Omega} |\nabla u|^2 \;dx+\int_{\partial B_r}\beta(x)u^2 \,d\,\mathcal{H}^{n-1}=\sigma_{\beta}(\Omega)\int_{\partial\Omega_0}u^2 \;d\mathcal{H}^{n-1}.
\end{equation}
Let us use the following function $z=u^2/(v+\varepsilon)$, for $\epsilon >0$,
as a test function in the definition \ref{WS} of weak solution for $v$. 
We get
\begin{equation}\label{testfuncv}
    \int_{\Omega} \bigg[\frac{2u\langle \nabla u,\nabla v\rangle}{v+\varepsilon}-\frac{u^2|\nabla v|^2}{(v+\varepsilon)^2}\bigg]\,dx +\int_{\partial B_r}\frac{v}{v+\varepsilon}\beta(x)u^2 \,d\,\mathcal{H}^{n-1}=\sigma\int_{\partial\Omega_0}\frac{v}{v+\varepsilon}u^2 \;d\mathcal{H}^{n-1}.
\end{equation}
If we subtract \eqref{testfuncu} by \eqref{testfuncv}, since $v/(v+\varepsilon)<1$, we get
\begin{equation*}
\begin{split}
    0\le \int_{\Omega}\abs{ \nabla u -\frac{u\nabla v}{v+\varepsilon}}^2\,dx = \int_{\Omega} \bigg[|\nabla u|^2-&\frac{2u\langle \nabla u,\nabla v\rangle}{v+\varepsilon}+\frac{u^2|\nabla v|^2}{(v+\varepsilon)^2}\bigg]\,dx\\ &\le \int_{\partial\Omega_0}\bigg[\sigma_{\beta}(\Omega)-\sigma\frac{v}{v+\varepsilon}\bigg]u^2\;d\mathcal{H}^{n-1}.
\end{split}
\end{equation*}
Passing to the limit as $\varepsilon\to 0$, we get 
\begin{equation*}
    [\sigma_{\beta}(\Omega)-\sigma]\int_{\partial\Omega_0}u^2\,d\mathcal{H}^{n-1}\ge 0.
\end{equation*}
Since $\sigma_{\beta}(\Omega)$ is the smallest eigenvalue, the only possibility is that $\sigma=\sigma_{\beta}(\Omega)$ and by the simplicity of the first eigenvalue, it must be $v=u$ up to multiplicative constants.
\end{proof}
\section{The first Steklov-Robin eigenvalue in the spherical shell}
Let us consider now $A_{r,R}=B_R\setminus\overline{B_r}$, where $B_R$ and $B_r$ are balls centered at the origin with radii $R>r>0$. Let $\beta(x)=\beta>0$ be positive and constant on $\partial B_r$ and let us consider the Steklov-Robin eigenvalue problem for the Laplacian in the spherical shell
\begin{equation} \label{SRP}
\begin{cases}
\Delta u = 0 & \text{in}\; A_{r,R}\\
\frac{\partial u}{\partial \nu}+ \beta u = 0 & \text{on}\; \partial B_r \\
\frac{\partial u}{\partial \nu} = \sigma u & \text{on}\; \partial B_R,
\end{cases}
\end{equation}
where $\nu$ is the outer unit normal to $\partial A_{r,R}$.\\
We are going to compute the solutions to problem \eqref{SRP}. 
\begin{thm}
The first Steklov-Robin eigenvalue of the problem \eqref{SRP} is 
\begin{equation}\label{Sigmaspherical}
   \sigma_{\beta}(A_{r,R}) \begin{cases}
    \frac{\ds1}{\ds\frac{R}{\beta r}+R\log \frac{R}{r}} & n=2\\
    \frac{\ds n-2}{\displaystyle \frac{\ds n-2}{\beta}\bigg(\frac{R}{r}\bigg)^{n-2}+ R\bigg[\bigg(\frac{R}{r}\bigg)^{n-2}-1\bigg]} & n\ge 3,
    \end{cases}
\end{equation}
and the corresponding eigenfunctions are the following
\begin{equation}\label{Eigspherical}
   u(x)= \begin{cases}
    \ds \log \frac{\abs{x}}{r}+ \frac{1}{\beta r} & n=2\\
    \ds \frac{1}{r^{n-2}}-\frac{1}{\abs{x}^{n-2}}+\frac{n-2}{\beta}\frac{1}{Rr^{n-2}} & n\ge 3.
    \end{cases}
\end{equation}
\end{thm}
\begin{proof}
Since the radial symmetry of the problem and the rotational invariance of the Laplacian, we look forward to a solution which is of the type $u(x)=v(\abs{x})=v(s)$, where $s=\abs{x}$. Computing the Laplacian of $v$ we get
\begin{equation*}
v''+ \frac{n-1}{s}v'=0,
\end{equation*}
which is equivalent to 
\begin{equation*}
(s^{n-1}v')'=0.
\end{equation*}
So integrating twice we get 
\begin{equation} \label{funsol}
v(s)=\begin{cases}
c_1 \log s + c_2 & n=2 \\
\displaystyle \frac{c_1}{s^{n-2}}+c_2 & n\ge 3.
\end{cases}
\end{equation}
We are going to find the solution to \eqref{SRP} by using the boundary conditions on $B_r$ and $B_R$. 
Let us begin by the bidimensional case. Using the boundary condition we get the following system in the unknown variables $c_1$ and $c_2$
\begin{equation} \label{system}
\begin{cases}
-\frac{c_1}{r}+ \beta (c_1 \log r+c_2)=0\\
\frac{c_1}{R}-\sigma (c_1 \log R+c_2)=0.
\end{cases}
\end{equation}
Since this is a homogeneous system, the only way not to have $c_1=c_2 = 0$ is that
\begin{equation*}
\det \begin{pmatrix}
-1/r+\beta\log r & \beta \\
1/R-\sigma \log R &-\sigma
\end{pmatrix}= 0.
\end{equation*}
From this we get that
\begin{equation*}
\sigma_{\beta}(A_{r,R})= \frac{1}{\frac{R}{\beta r}+R\log \frac{R}{r}}.	
\end{equation*}
Since this choice of $\sigma$, $c_1$ and $c_2$ must be linearly dependents. Hence if we chose $c_1=1$, by using the second equation in \eqref{system} we have that
\begin{equation*}
c_2 = \frac{1}{\sigma R}-\log R= \frac{1}{\beta r}+ \log\frac{R}{r}-\log R= \frac{1}{\beta r}-\log r.
\end{equation*}
Hence inserting $c_1$ and $c_2$ in \eqref{funsol}, we have
\begin{equation*}
u(x)= \log \frac{\abs{x}}{r}+ \frac{1}{\beta r}.
\end{equation*}
In higher dimensions the system becomes
\begin{equation} \label{system2}
\begin{cases}
\frac{n-2}{r^{n-1}}c_1+ \beta (\frac{c_1}{r^{n-2}}+c_2 )=0\\
-\frac{n-2}{R^{n-1}}c_1- \sigma (\frac{c_1}{R^{n-2}}+c_2 )=0.
\end{cases}
\end{equation}
Proceding in the same way as before, we find that
\begin{equation*}
\sigma=\sigma_{\beta}(A_{r,R})= \frac{n-2}{\displaystyle \frac{n-2}{\beta}\bigg(\frac{R}{r}\bigg)^{n-2}+ R\bigg[\bigg(\frac{R}{r}\bigg)^{n-2}-1\bigg]}.
\end{equation*}
Hence chosing $c_1=-1$
\begin{equation*}
c_2 = \frac{1}{R^{n-2}}+\frac{n-2}{\sigma R^{n-1}}=\frac{1}{r^{n-2}}+ \frac{n-2}{\beta}\frac{1}{Rr^{n-2}},
\end{equation*}
and 
\begin{equation*}
u(x)= \frac{1}{r^{n-2}}-\frac{1}{\abs{x}^{n-2}}+\frac{n-2}{\beta}\frac{1}{Rr^{n-2}}.
\end{equation*}
Eventually, in any dimension, with these choices of the constants $c_1,c_2$, the corresponding eigenfunctions do not change sign and so they must be eigenfunctions corresponding to the first Steklov-Robin eigenvalue $\sigma_{\beta}(A_{r,R})$.
\end{proof}

By the explicit form of $\sigma_{\beta}(A_{r,R})$ and the corresponding eigenfunctions in \eqref{Sigmaspherical}-\eqref{Eigspherical}, we deduce the following properties
when we let vary $\beta$ or the radii of the spherical shell.
\begin{itemize}
    \item $\lim_{\beta \to 0} \sigma_{\beta}(A_{r,R}) =0$, and in particular
\begin{equation}\label{asym2}
    \lim_{\beta \to 0} \frac{\sigma_{\beta}(A_{r,R})}{\beta}=\frac{P(B_r)}{P(B_R)}.
\end{equation}
\item Recalling the explicit value of the first Steklov-Dirichlet eigenvalue of spherical shells (see \cite{verma2020eigenvalue, gavitone2021isoperimetric,paoli2020stability}), we have
\begin{equation}\label{asym3}
    \lim_{\beta \to \infty}\sigma_{\beta}(A_{r,R})=\sigma_D(A_{r,R}).
\end{equation}
\item Finally we have 
\begin{equation}\label{asym4}
    \lim_{r\to 0}\sigma_{\beta}(A_{r,R})=\lim_{R\to 0}\sigma_{\beta}(A_{r,R})=0,
\end{equation}
\end{itemize}

\vspace{0.5cm}
We will see in the next section that all of these behaviours will persist in the case of a generic $\Omega$.

\section{Asymptotic estimates of $\sigma_{\beta}(\Omega)$ with respect to $\beta$ and $r$}
In this section we will study the behaviour of $\sigma_{\beta}(\Omega)$ when $m$ (or $\beta$ constant) and $r$ vary.
\subsection{Behaviour with respect to the inner radius}
We will prove that \eqref{asym4} continues to hold for a general annular domain $\Omega$ by proving some suitable estimates in terms of the radius of the hole. Indeed let us consider the spherical shell $A_{r,R_m}$, where $R_m$ is defined in \eqref{re}, which is contained in $\Omega$. If we choose as a test function in the variational characterization of $\sigma_{\beta}(\Omega)$ 
\begin{equation*}
    \varphi = 
    \begin{cases}
    \ds v(|x|) & \text{in} \, A_{r,R_m}\\
    \ds v(R_m) & \text{in}\, \Omega \setminus \overline{A_{r,R_m}},
    \end{cases}
\end{equation*}
where $v$ is the first eigenfunction in $A_{r,R_m}$ with constant Robin parameter $\overline{\beta}=1/ m$, then 
\begin{equation*}
    \sigma_{\beta}(\Omega)\le \sigma_1(A_{r,R_m}).
\end{equation*}
As a consequence when $r\to 0$,
\begin{equation*}
    \sigma_{\beta}(\Omega) \to 0.
\end{equation*}
A natural question, now, is asking if we have a lower bound in terms of the first Steklov-Robin eigenvalue of an opportune spherical shell. In the following result we prove that for starshaped set is possible to have an optimal lower bound $\sigma_{\beta}(\Omega)$.
\begin{thm}\label{lowbound}
\label{lower}
Let $r>0$ and $\Omega_0\subset\R^n$ be an open, bounded starshaped set such that  $\overline{B_{r}}\subset\Omega_0$ and  let $\Omega=\Omega_0\setminus \overline{B_r}$.  Then, if $\tilde{\beta} = \inf_{x\in \partial B_r}\beta(x)$, it holds
\begin{equation}
\label{lb}
\sigma_{\beta}(\Omega)\ge \ds \frac{\ds \sigma_{\tilde{\beta}} \left(A_{r,R_m}\right)}{R_M^{n-1} \max_{\mathbb S^{n-1}} \left(\sqrt{1+ \frac{|\nabla_\tau \rho_0|^2}{\rho_0^2}}\right)}  ,
\end{equation}
where $R_m$ and $R_M$ are defined in \eqref{re}, $\rho_0$ is the radial function of $\Omega_0$ defined in \eqref
{ro} and $A_{r,R_m}$ is the spherical shell with radii $r$ and $R_m$.\\
Moreover, the equality case holds if and only if $\Omega_0$ is also a ball $B_R$ centered at the origin of radius $R>0$.

\end{thm} 
\begin{proof}
 We will  follow an idea used in \cite{kuttler1968lower} for the planar case  and in \cite{garcia2015lower, verma2018bounds} for any dimension.
Let $u\in H^1(\Omega)$
By using spherical  coordinates and the notation introduced in Section $2$:
\[
\partial \Omega_0 = \{x\,\rho_0(x),  x \in \mathbb S^{n-1}\}
\]
and
\begin{equation*}
    \partial B_r = \{xr, \, x \in \mathbb{S}^{n-1}\}.
\end{equation*}
The integrals over the boundaries $\partial \Omega_0$ and $\partial B_R$ of $u^2$ become
\begin{equation*}
\ds\int_{\partial\Omega_0}u^2\;d\mathcal{H}^{n-1}= \ds \int_{\mathbb S^{n-1}} u^2 \,\,\sqrt{1+\left(\frac{|\nabla_\tau \rho_0|}{\rho_0}\right)^2}\,(\rho_0)^{n-1}\,d\mathcal H^{n-1},
\end{equation*}
\begin{equation} \label{numball}
    \int_{\partial B_r}\beta(x)u^2 \,d\mathcal{H}^{n-1}\ge \tilde{\beta}r^{n-1}\int_{\mathbb{S}^{n-1}} u^2\,\mathcal{H}^{n-1}.
\end{equation}
In particular we have
\begin{equation}
\label{den}
 \ds\int_{\partial\Omega_0}u^2\;d\mathcal{H}^{n-1} \le (R_M)^{n-1} \max_{\mathbb S^{n-1}} \left(\sqrt{1+ \frac{|\nabla_\tau \rho_0|^2}{\rho_0^2}}\right) \int_{\mathbb S^{n-1}} u^2 \,\,d\mathcal H^{n-1}.
\end{equation}
We can parametrize
\[
\Omega=\{s\in \mathbb R^n \colon s=x\,\tilde{r} ,\, x \in \mathbb S^{n-1},\, \tilde{r}\le \tilde{r}\le \rho_0(x)  \} ,
\]
by using spherical coordinates, where we denote by $R(y)=\rho_0(x(y))$, and  $x\colon y\in U\subset \mathbb R^{n-1}\to x(y) \in \mathbb S^{n-1}$ is a standard parametrization of the boundary of the unit ball in $\mathbb R^{n}$. Then we get
\begin{equation}
\label{den1}
\ds\int _{\Omega}|\nabla u|^2\;ds=\displaystyle \int_{U} \int_{r}^{R(y)}\left\{\left(\frac{\partial u}{\partial \tilde{r}}\right)^2+\frac{1}{\tilde{r}^2}|\nabla_{\tau} u|^2\right\}\tilde{r}^{n-1}\sqrt{\tilde g}\, d\tilde{r}\,dy,\\
\end{equation}
where  $\sqrt {\tilde g}$ is the determinant of the matrix $\tilde g_{ij}$, that is the standard  metric on $\mathbb S^{n-1}$ and $\nabla_{\tau} u$ is the component of $\nabla u$ tangential to $\mathbb S^{n-1}$. Therefore
\begin{equation}
\label{num}
\displaystyle 
\ds\int _{\Omega}|\nabla u|^2\;ds\ge\int_{U} \int_{r}^{R_m}\left\{\left(\frac{\partial u}{\partial \tilde{r}}\right)^2+\frac{1}{\tilde{r}^2}|\nabla_{\tau} u|^2\right\}\tilde{r}^{n-1}\sqrt{\tilde g}\, d\tilde{r}\,dy,
\end{equation}
Combining \eqref{den}, \eqref{numball} and \eqref{num} and recalling \eqref{varcar},  we get
\begin{multline}
\sigma_{\beta}(\Omega)\ge \ds \frac{\displaystyle \int_{U} \int_{r}^{R_m}\left\{\left(\frac{\partial u}{\partial \tilde{r}}\right)^2+\frac{1}{\tilde{r}^2}|\nabla_{\tau} u|^2\right\}\tilde{r}^{n-1}\sqrt{\tilde g}\, d\tilde{r}\,dy+ \tilde{\beta}r^{n-1}\int_{\mathbb{S}^{n-1}}u^2\,d\mathcal{
H}^{n-1}
}{\ds(R_M)^{n-1} \max_{\mathbb S^{n-1}} \left(\sqrt{1+ \frac{|\nabla_\tau \rho_0|^2}{\rho_0^2}}\right) \ds \int_{\mathbb S^{n-1}} u^2 \,\,d\mathcal H^{n-1}}
\ge\\ \frac{\sigma_{\tilde{\beta}}(A_{r,R_m})}{\ds(R_M)^{n-1} \max_{\mathbb S^{n-1}} \left(\sqrt{1+ \frac{|\nabla_\tau \rho_0|^2}{\rho_0^2}}\right)}.
\end{multline}   

 Finally, we stress that  the equality case implies that  all the inequalities become equalities. So, we have that  $\nabla_{\tau}\rho_0=0$ and  $\rho_0(x)=\overline{R}$, with $\overline{R}>R$ constant.
 \end{proof}

Theorem \ref{lowbound} tells us that as long $\Omega_0$ is an open, bounded starshaped set, and $r>0$, then $\sigma_{\beta}(\Omega)$ remains away from zero. Is this property still true for any open, bounded set in $\mathbb{R}^n$ with Lipschitz boundary? In general the answer is no, as showed in the following bidimensional counterexample, that is contained in \cite{girouard2017spectral}, that can be easily generalized in any dimension.
\begin{cntex}
Let $\beta(x)=\beta$ be a positive constant and let us consider a sequence of open, bounded and connected sets $\{\Omega_\varepsilon\}\subset \mathbb{R}^2 $ as follows
\begin{equation*}
	\Omega_\varepsilon = B(x_0)\cup R_{\varepsilon}\cup (B(x_1)\setminus \overline{B_r(x_1)}).
\end{equation*}
Here 
\begin{equation*}
R_{\varepsilon}= \bigg(-\frac{\varepsilon}{2},\frac{\varepsilon}{2}\bigg)\times \bigg(-\frac{\varepsilon^3}{2},\frac{\varepsilon^3}{2}\bigg)
\end{equation*}
is a rectangle centered at the origin with sides of length $\varepsilon$ and $\varepsilon^3$ respectively, $B(x_1)$, $B(x_2)$ are two dimensional balls of radius $1$ centered at the points $x_1$ and $x_2$, chosen such that the rectangle $R_{\varepsilon}$ is well glued and eventually $B_r(x_2)$ is a concentric ball in $B(x_2)$ of radius $0<r<1$.\\

Let us consider the following function
\begin{equation*}
	u(x,y)= \begin{cases}
	\sin \bigg(\ds\frac{2\pi x}{\varepsilon}\bigg) & \text{in} \, R_{\varepsilon}\\
	0 & \text{elsewhere},
	\end{cases}
\end{equation*}
which is a continuous test function for the first Steklov-Robin eigenvalue.\\
Let us evaluate $u$ in the numerator and denominator the Rayleigh quotient. The denominator becomes
\begin{equation*}
\begin{split}
\int_{\partial \Omega_\varepsilon}u^2 \,d\mathcal{H}^{n-1}&= \int_{\partial R_\varepsilon}u^2\,d\mathcal{H}^{n-1} = 2\int_{-\frac{\varepsilon}{2}}^{\frac{\varepsilon}{2}}\sin^2 \bigg(\ds\frac{2\pi x}{\varepsilon}\bigg)\,dx\\
&=4\int_0^{\frac{\varepsilon}{2}} \sin^2\bigg(\ds\frac{4\pi x}{\varepsilon}\bigg)\,dx
=4\int_0^{\frac{\varepsilon}{2}}\frac{1-\cos\bigg(\ds\frac{4\pi x}{\varepsilon}\bigg)}{2}\,dx=\varepsilon
\end{split}
\end{equation*}
Since
\begin{equation*}
	|\nabla u|^2 = \bigg(\frac{\partial u}{\partial x}\bigg)^2 = \bigg(\frac{2\pi}{\varepsilon}\bigg)^2 \cos^2\bigg(\ds\frac{2\pi x}{\varepsilon}\bigg),
\end{equation*}
we have that the numerator is
\begin{equation*}
\begin{split}
\int_{\Omega_\varepsilon}|\nabla u|^2 \,dx\,dy&= \int_{R_\varepsilon}|\nabla u|^2\,dx\,dy = \bigg(\frac{2\pi }{\varepsilon}\bigg)^2\int_{-\frac{\varepsilon^3}{2}}^{\frac{\varepsilon^3}{2}}\int_{-\frac{\varepsilon}{2}}^{\frac{\varepsilon}{2}} \cos^2\bigg(\ds\frac{4\pi x}{\varepsilon}\bigg)\,dx\,dy \\
&= 2\bigg(\frac{2\pi }{\varepsilon}\bigg)^2\varepsilon^3\int_0^{\frac{\varepsilon}{2}}\frac{1+\cos\bigg(\ds\frac{4\pi x}{\varepsilon}\bigg)}{2}\,dx=2\pi^2\varepsilon^2.
\end{split}
\end{equation*}
In this way, since $u$ is zero on $\partial B_r(x_2)$, we get
\begin{equation*}
	\sigma_{\beta}(\Omega_\varepsilon)\le \frac{\ds \int_{\Omega_\varepsilon}|\nabla u|^2\,dx}{\ds \int_{\partial \Omega_\varepsilon}u^2 \,d\mathcal{H}^{n-1}}= 2\pi^2\varepsilon,
\end{equation*}
and
\begin{equation*}
	\sigma_{\beta}(\Omega_\varepsilon)\to 0
\end{equation*}
as $\varepsilon\to 0$. We stress the fact that the same proof can be exhibited even in the Steklov-Dirichlet case.
\end{cntex}

This counterexample gives us two information. The first is the one we already mentioned: if $\Omega_0$ is not starshaped, the first eigenvalue could be arbitrarily close to zero. The second one is that when $\Omega$ is not connected, then $\sigma_{\beta}(\Omega)$ could be zero, even though $r>0$. 

\subsection{Behaviour with respect to $m$ and $\beta$}
In this section we will give the proof of the Theorem \ref{mainstime}. 
\begin{proof}[Proof of Theorem \ref{mainstime}]
Firstly we prove inequality \eqref{Kutt_mi}. We observe that  the claim is well posed since,  by proceeding analogously as in the existence theorem, $\mu_1(\Omega)$  is positive. 
For any $w \in H^{1}(\Omega)$, for simplicity, we will use the following notation
\begin{equation}
\label{D}
D(w):=\displaystyle \int_{\Omega}|\nabla w|^2\,dx.
\end{equation}
Let $u$ be a positive eigenfunction corresponding to  $\sigma_{\beta}(\Omega)$. By the Minkowski   inequality and the definition of $\mu_1(\Omega)$ we have
\[
\sqrt{\int_{\partial\Omega_0}u^2\,d\mathcal H^{n-1}}\leq \sqrt{\int_{\partial\Omega_0}(u-c)^2 \,d\mathcal H^{n-1}}+\sqrt{c^2 P(\Omega_0)}\leq \sqrt{\frac{D(u)}{\mu_1(\Omega)}}+\sqrt{c^2 P(\Omega_0)},
\]
where $c$ is
\begin{equation}\label{c}
c=\dfrac{1}{m} \int_{\partial B_r}\beta(x) u \,d \mathcal{H}^{n-1},
\end{equation}
Squaring and using the arithmetic-geometric mean inequality, we have
\begin{equation}\label{A}
\begin{split}
\int_{\partial\Omega_0}u^2d\mathcal{H}^{n-1} &\leq \frac{D(u)}{\mu_1(\Omega)}+c^2 P(\Omega_0)+2\sqrt{\frac{D(u)c^2P(\Omega_0)}{\mu_1(\Omega)}}\\
&= D(u)\left(\frac{1}{\mu_1(\Omega)}+\frac{P(\Omega_0)}{m}\right)+c^2m\left(\frac{1}{\mu_1(\Omega)}+\frac{P(\Omega_0)}{m}\right)\\
&=\left(\frac{1}{\mu_1(\Omega)}+\frac{P(\Omega_0)}{m}\right)(D(u)+c^2 m).
\end{split}
\end{equation}
Applying H\"{o}lder inequality in \eqref{c} we have
\begin{equation*}
    c^2 m = \frac{1}{m}\bigg( \int_{\partial B_r}\beta(x)u\,d\mathcal{H}^{n-1}\bigg)^2\le \int_{\partial B_r}\beta(x)u^2\,d\mathcal{H}^{n-1},
\end{equation*}
so that
\begin{equation}
\begin{split}
\int_{\partial\Omega_0}u^2d\mathcal{H}^{n-1} 
&\le  \left(\frac{1}{\mu_1(\Omega)}+\frac{P(\Omega_0)}{m}\right) \left(D(u)+  \int_{\partial B_r} \beta (x)u^2 \,d\mathcal{H}^{n-1}\right)\\
&= \left(\frac{1}{\mu_1(\Omega)}+\frac{P(\Omega_0)}{m}\right)\left(\sigma_{\beta}(\Omega)\int_{\partial \Omega_0}u^2 d\mathcal{H}^{n-1}\right),
\end{split}
\end{equation}
which gives \eqref{kutmu}.\\
Now we prove inequality \eqref{Kutt_40i}. Let us stress the fact that the quantity $q(\Omega)$ defined in \eqref{qomega} is strictly positive. Indeed if $q(\Omega)=0$, then it would follow that $w=0$ a.e. on $\partial B_r$ and since $\partial w\backslash \partial \nu=0$ on $\partial \Omega_0$, then it would imply $w\equiv 0$ in $\Omega$, which is a contradiction. Let $u$ be the eigenfunction corresponding to $\sigma_{\beta}(\Omega)$, solution to problem \eqref{proSR}. Let us observe that $u=v+h$, where $v$ and $h$ solve the following problems
\begin{equation*}
    \begin{cases}
    \Delta v = 0 & \text{in}\,\Omega\\
    v=0 & \text{on}\,\partial B_r\\
    \frac{\de v}{\de \nu}=\frac{\de u}{\de \nu} & \text{on}\,\partial \Omega_0,
    \end{cases}
    \qquad \qquad
    \begin{cases}
    \Delta h = 0 & \text{in}\,\Omega\\
    h=u & \text{on}\,\partial B_r\\
    \frac{\de h}{\de \nu}=0 & \text{on}\,\partial \Omega_0.
    \end{cases}
\end{equation*}
It is easy to check that
\begin{equation}\label{gradientvh}
    \int_{\Omega}\abs{\nabla u}^2\,dx = \int_{\Omega}\abs{\nabla v}^2\,dx+\int_{\Omega}\abs{\nabla h}^2\,dx.
\end{equation}
Then, proceeding as the proof of inequality \eqref{kutmu},  
by applying Minkowski inequality and using  \eqref{gradientvh}, \eqref{testfuncv} and \eqref{q1}, we get
\begin{equation*}
    \begin{split}
\sqrt{\ds\int_{\partial\Omega_0}u^2\,\mathcal{H}^{n-1}}&\le \sqrt{\ds\int_{\partial\Omega_0}v^2\,\mathcal{H}^{n-1}}+\sqrt{\ds\int_{\partial\Omega_0}h^2\,\mathcal{H}^{n-1}}\\
    & \le \sqrt{\frac{1}{\sigma_D(\Omega)}\ds D(u)}+\sqrt{\frac{1}{q_\beta(\Omega)}\ds\int_{\partial B_r}\beta(x)u^2\,\mathcal{H}^{n-1}}.
\end{split}
\end{equation*}
Squaring both sides we have
\begin{equation*}
    \begin{split}
    \int_{\partial \Omega_0}u^2\,d\mathcal{H}^{n-1}\le \frac{D(u)}{\sigma_D(\Omega)}&+ \frac{1}{q_\beta(\Omega)}\int_{\partial B_r}\beta(x)u^2\,\mathcal{H}^{n-1}\\
    &+2\sqrt{\frac{D(u)}{\sigma_D(\Omega)}\frac{1}{q_\beta(\Omega)}\int_{\partial B_r}\beta(x)u^2\,\mathcal{H}^{n-1}}
    \end{split}
\end{equation*}
Applying the arithmetic-geometric mean inequality and H\"{o}lder inequality we get
\begin{equation*}
    \begin{split}
     \int_{\partial \Omega_0}u^2\,d\mathcal{H}^{n-1} &\le \frac{D(u)}{\sigma_D(\Omega)}+ \frac{D(u)}{q_\beta(\Omega)}\\
     &+\frac{1}{q_\beta(\Omega)}\int_{\partial B_r}\beta(x)u^2\,\mathcal{H}^{n-1}
     + \frac{1}{\sigma_D(\Omega)}\int_{\partial B_r}\beta(x)u^2\,\mathcal{H}^{n-1}\\
     &= \bigg(\frac{1}{\sigma_D(\Omega)}+\frac{1}{ q_\beta(\Omega)}\bigg)\bigg(D(u)+\int_{\partial B_r}\beta(x)u^2\,\mathcal{H}^{n-1}\bigg)
     \end{split}
\end{equation*}
This gives \eqref{Kutt_40i}.

\end{proof}

\begin{rem}
We stress that a rough but meaningful estimate can be obtained 
choosing as a test function in \eqref{varcar} the constant function. In this case we have the following  upper bound
\begin{equation}\label{roughestimate}
	\sigma_{\beta}(\Omega)\le \frac{m}{P(\Omega_0)},
\end{equation}
which immediately gives that  when $m \to 0$, then $\sigma_{\beta}(\Omega) \to 0$.
However, inequality \eqref{kutmu} allows us to show that it holds 
\begin{equation}
\label{asi}
\lim_{m \rightarrow0} \frac{\sigma_{\beta}(\Omega)}{m}=\frac{1}{P(\Omega_0)}.
\end{equation}
Indeed  we have
\begin{equation}
\label{s1}
\frac{1}{P(\Omega_0)} \ge \frac{\sigma_{\beta}(\Omega)}{m}\ge \dfrac{\mu_1(\Omega)}{m + P(\Omega_0)\mu_1(\Omega)},
\end{equation}
where the first inequality follows by \eqref{roughestimate} and the second by using \eqref{Kutt_mi}. 
Taking in  \eqref{s1} the limit for $m$ which goes to zero  one gets \eqref{asi}.\\
In particular, if $\beta(x)=\beta$ is the constant function we recover the asymptotic formula valid in the radial case
$$
\lim_{\beta \to 0} \frac{\sigma_{\beta}(\Omega)}{\beta}= \frac{P(B_r)}{P(\Omega_0)}.
$$
\end{rem}


\begin{rem}
We observe that if $\beta(x)=\beta$ is constant, then inequality \eqref{Kutt_40i} gives 
\[
\frac{1}{\sigma_{\beta}(\Omega)}-\frac{1}{\sigma_D(\Omega)} \le \frac{1}{ \beta q(\Omega)}
\]
which immediately  implies that
\begin{equation*}
    \lim_{\beta \to \infty} \sigma_{\beta}(\Omega)= \sigma_D(\Omega).
\end{equation*}
\end{rem}

\section{Asymptotic behaviour of the eigenfunctions}
In the radial case it is easy to check that the eigenfunctions corresponding to $\sigma_\beta (\Omega)$ converge pointwise to the corresponding eigenfunction of the Steklov-Dirichlet eigenvalue. In this section we will prove that the same happens even for a generic $\Omega$ when $\beta$ is a positive constant.
\begin{proof}[Proof of Theorem \ref{mainstime}]
Without loss of generality, we suppose that $u_\beta$ and $v$ are normalized in $L^2(\partial \Omega_0)$ with unitary norm.
By their weak formulations we get
\begin{equation} \label{WSR}
\int_{\Omega} \langle\nabla u,\nabla \varphi\rangle \;dx+\beta\int_{\partial B_r}u\varphi \,d\,\mathcal{H}^{n-1}=\sigma_\beta(\Omega)\int_{\partial\Omega_0}u \varphi \;d\mathcal{H}^{n-1}, \qquad \forall \varphi \in H^1(\Omega),
\end{equation}
and
\begin{equation} \label{WSD}
\int_{\Omega} \langle\nabla v,\nabla \varphi\rangle \;dx=\sigma_D(\Omega)\int_{\partial\Omega_0}v\varphi \;d\mathcal{H}^{n-1}, \qquad \forall \varphi \in H^1_{\partial B_r}(\Omega).
\end{equation}
In particular $v$ is an admissible function for \eqref{WSR}, hence
\begin{equation}\label{vinubeta}
    \int_{\Omega} \langle\nabla u,\nabla v\rangle \;dx=\sigma_\beta(\Omega)\int_{\partial\Omega_0}u v \;d\mathcal{H}^{n-1}.
\end{equation}
Now, by Friederich's inequality \eqref{friedin}, there exists a positive constant $C>0$ such that
\begin{equation} \label{L2estimate}
    \|u_\beta - v\|_{L^2(\Omega)}\le C\bigg(\|\nabla u_\beta -\nabla v\|_{L^2(\Omega)}+\|u_\beta -v\|_{L^2(\partial \Omega)}\bigg).
\end{equation}
Let us now estimate the terms inside the round brackets. Using the variational characterization of $u_\beta$ and $v$, the fact that $\|u_\beta\|_{L^2(\partial \Omega_0)}=\|v\|_{L^2(\partial \Omega_0)}=1$ and equation \eqref{vinubeta} we get
\begin{equation*}
\begin{split}
    \|\nabla u_\beta &-\nabla v\|_{L^2(\Omega)}+\|u_\beta -v\|_{L^2(\partial \Omega)}= \\
    &=  \int_{\Omega}|\nabla u_\beta|^2\,dx -2\int_{\Omega}\langle\nabla u_\beta,\nabla v \rangle\,dx + \int_{\Omega}|\nabla v|^2\,dx+\\
    &+\int_{\partial\Omega_0}|u_\beta-v|^2\,d\mathcal{H}^{n-1}+\int_{\partial B_r}u_\beta^2\,d\mathcal{H}^{n-1} \\
    &=(1-\beta)\int_{\partial B_r}u^2_\beta\,d\mathcal{H}^{n-1} +\int_{\partial\Omega_0}|u_\beta-v|^2\,d\mathcal{H}^{n-1}\\
    &+\sigma_\beta(\Omega)-2\sigma_\beta(\Omega)\int_{\partial \Omega_0}uv\,d\mathcal{H}^{n-1}+\sigma_D(\Omega).
\end{split}
\end{equation*}
Since we want $\beta \to +\infty$, we can chose $\beta>1$. Therefore
\begin{equation*}
\begin{split}
    \|\nabla u_\beta &-\nabla v\|_{L^2(\Omega)}+\|u_\beta -v\|_{L^2(\partial \Omega)}<\\
    &<\int_{\partial\Omega_0}|u_\beta-v|^2\,d\mathcal{H}^{n-1}+\sigma_\beta(\Omega)-2\sigma_\beta(\Omega)\int_{\partial \Omega_0}uv\,d\mathcal{H}^{n-1}+\sigma_D(\Omega).
    \end{split}
\end{equation*}
Since in the previous section we proved that $\sigma_\beta(\Omega)\to \sigma_D(\Omega)$ as $\beta \to +\infty$, if we prove that $\|u_\beta-v\|_{L^2(\partial \Omega_0)}\to 0$ we have \eqref{H1conv}.\\
Let $\{\beta_k\}_{k\in \mathbb{N}}$ a sequence of Robin parameters such that $\beta_k>1$ for every $k\ge 1$ and such that $\beta_k\to +\infty$. Let $u_k\equiv u_{\beta_k}$ be the eigenfunctions corresponding to the eigenvalues $\sigma_{\beta_k}(\Omega)$, such that $\|u_k\|_{L^2(\partial \Omega_0)}=1$. We already know that $\sigma_{\beta_k}(\Omega)\to \sigma_D(\Omega)$ as $k\to +\infty$.\\
Applying once more Friederich's inequality \eqref{friedin}, we get
\begin{equation*}
    \begin{split}
        \|u_k\|_{H^1(\Omega)} &= \|\nabla u_k\|_{L^2(\Omega)} + \|u\|_{L^2(\Omega)}\\
        &\le \|\nabla u_k\|_{L^2(\Omega)} + C(\|\nabla u_k\|_{L^2(\Omega)}+\|u\|_{L^2(\partial\Omega)})\\
        &\le \|\nabla u_k\|_{L^2(\Omega)} + C(\|\nabla u_k\|_{L^2(\Omega)}+\beta_k\|u\|_{L^2(\partial B_r)}+1)\\
        &\le \sigma_D(\Omega) + C(\sigma_D(\Omega)+1)= (C+1)\sigma_D(\Omega)+ C.
    \end{split}
\end{equation*}
Hence we have that the sequence $\{u_k\}_{k\in\mathbb{N}}$ is uniformly bounded in $H^1(\Omega)$ with respect to $\beta_k$. So there exists a function $\Bar{u}\in H^1(\Omega)$ such that $\nabla u_k \to\nabla \Bar{u}$ weakly in $L^2(\Omega)$ and by the compactness of the trace operator we have that $u_k\to \Bar{u}$ in $L^2(\partial \Omega)$ and hence $u_k\to \Bar{u}$ a.e. on $\partial \Omega$. In particular, since for every $k\in\mathbb{N}$ 
\begin{equation*}
    \beta_k \|u_k\|_{L^2(\partial B_r)}\le \sigma_D(\Omega)<+\infty,
\end{equation*}
it is necessary that $\|u_k\|_{L^2(\partial B_r)}\to 0$ as $k \to +\infty$. Using Fatou's Lemma we have that for $k\to +\infty$
\begin{equation*}
    0\le \int_{\partial B_r} \Bar{u}^2\,d\mathcal{H}^{n-1}\le \liminf_{k\to +\infty}\int_{\partial B_r}u_k^2\,d\mathcal{H}^{n-1}\to 0.
\end{equation*}
Therefore $\|\Bar{u}\|_{L^2(\partial \Omega)}=0$, and hence $\Bar{u}\in H^1_{\partial B_r}(\Omega)$. We only need to show that $\Bar{u}=v$. If we use $\Bar{u}$ as a test function in \eqref{WSR} we have that
\begin{equation*}
    \int_{\Omega} \langle\nabla u_k,\nabla \Bar{u}\rangle \;dx=\sigma_\beta(\Omega)\int_{\partial\Omega_0}u_k \Bar{u} \;d\mathcal{H}^{n-1}.
\end{equation*}
Now by the weak convergence of the gradients, the weak convergence of $u_k$ on $L^2(\partial \Omega_0)$ and since $\sigma_{\beta_k}\to \sigma_D$, we get 
\begin{equation*}
    \int_{\Omega}|\nabla \Bar{u}|^2\,dx = \sigma_D(\Omega) \int_{\partial \Omega_0}\Bar{u}^2\,d\mathcal{H}^{n-1}.
\end{equation*}
By the simplicity of $\sigma_D$ and the fact that $\|v\|_{L^2(\partial \Omega_0)}=\|\Bar{u}\|_{L^2(\partial \Omega_0)}=1$, it follows that $\Bar{u}=v$. This implies that $\|u_{\beta_k}-v\|_{L^2(\partial\Omega_0)}\to 0$ and so
\begin{equation*}
    \|u_\beta - v\|_{L^2(\Omega)}\to 0.
\end{equation*}
In particular the following
\begin{equation*}
\begin{split}
    0&\le \|u_\beta - v\|_{L^2(\Omega)}\le C(\|\nabla u_\beta - \nabla v\|_{L^2(\Omega)}+\|u_\beta - v\|_{L^2(\partial\Omega)})\\
    &\le C\bigg(\sigma_\beta(\Omega)-2\sigma_\beta(\Omega)\int_{\partial \Omega_0}uv\,d\mathcal{H}^{n-1}+\sigma_D(\Omega)\bigg)\to 0
\end{split}
\end{equation*}
implies that $\|\nabla u_\beta -\nabla v\|_{L^2(\Omega)}\to 0$. This conclude the proof.
\end{proof}

\section*{Acknowledgements}
This work has been partially supported by the MiUR-PRIN 2017 grant \lq\lq Qualitative and quantitative aspects of nonlinear PDEs\rq\rq, by GNAMPA of INdAM and by FRA 2020 \lq\lq Optimization problems in Geometric-functional inequalities and nonlinear PDEs\rq\rq(OPtImIzE).
\bibliographystyle{plain}
\bibliography{biblio}

\begin{thebibliography}{10}

\bibitem{cianchi2016sobolev}
Andrea Cianchi and Vladimir Maz'ya.
\newblock Sobolev inequalities in arbitrary domains.
\newblock {\em Advances in Mathematics}, 293:644--696, 2016.

\bibitem{di2022two}
Giuseppina Di~Blasio and Nunzia Gavitone.
\newblock Two inequalities for the first robin eigenvalue of the finsler
  laplacian.
\newblock {\em Archiv der Mathematik}, 118(2):205--213, 2022.

\bibitem{dittmar2005eigenvalue}
Bodo Dittmar.
\newblock Eigenvalue problems and conformal mapping, r. k{\"{}} uhnau (ed.),
  handbook of complex analysis: Geometric function theory. vol. 2, 2005.

\bibitem{evans2010partial}
Lawrence~C Evans.
\newblock {\em Partial differential equations}, volume~19.
\newblock American Mathematical Soc., 2010.

\bibitem{friedrichs1928randwert}
Kurt Friedrichs.
\newblock Die randwert-und eigenwertprobleme aus der theorie der elastischen
  platten.(anwendung der direkten methoden der variationsrechnung).
\newblock {\em Mathematische Annalen}, 98(1):205--247, 1928.

\bibitem{ftouh2022place}
Ilias Ftouh.
\newblock Where to place a spherical obstacle so as to maximize the first
  nonzero steklov eigenvalue.
\newblock {\em ESAIM: Control, Optimisation and Calculus of Variations}, 28:6,
  2022.

\bibitem{garcia2015lower}
Gonzalo Garcia and Oscar Montano.
\newblock A lower bound for the first steklov eigenvalue on a domain.
\newblock {\em Revista Colombiana de Matem{\'a}ticas}, 49(1):95--104, 2015.

\bibitem{gavitone2021isoperimetric}
Nunzia Gavitone, Gloria Paoli, Gianpaolo Piscitelli, and Rossano Sannipoli.
\newblock An isoperimetric inequality for the first steklov-dirichlet laplacian
  eigenvalue of convex sets with a spherical hole.
\newblock {\em to appear on Pacific Journal of Mathematics}, 2021.

\bibitem{gavitone2021monotonicity}
Nunzia Gavitone and Gianpaolo Piscitelli.
\newblock A monotonicity result for the first steklov-dirichlet laplacian
  eigenvalue.
\newblock {\em arXiv preprint arXiv:2111.03385}, 2021.

\bibitem{girouard2017spectral}
Alexandre Girouard and Iosif Polterovich.
\newblock Spectral geometry of the steklov problem (survey article).
\newblock {\em Journal of Spectral Theory}, 7(2):321--359, 2017.

\bibitem{hong2020shape}
Jiho Hong, Mikyoung Lim, and Dong-Hwi Seo.
\newblock Shape monotonicity of the first steklov-dirichlet eigenvalue on
  eccentric annuli.
\newblock {\em arXiv preprint arXiv:2007.10147}, 2020.

\bibitem{kuttler1973note}
James~R Kuttler.
\newblock A note on a paper of sperb.
\newblock {\em Zeitschrift f{\"u}r angewandte Mathematik und Physik ZAMP},
  24(3):431--434, 1973.

\bibitem{kuttler1968lower}
JR~Kuttler and VG~Sigillito.
\newblock Lower bounds for stekloff and free membrane eigenvalues.
\newblock {\em SIAM Review}, 10(3):368--370, 1968.

\bibitem{maggi2012sets}
Francesco Maggi.
\newblock {\em Sets of finite perimeter and geometric variational problems: an
  introduction to Geometric Measure Theory}.
\newblock Number 135. Cambridge University Press, 2012.

\bibitem{maz2013sobolev}
Vladimir Maz'ya.
\newblock {\em Sobolev spaces}.
\newblock Springer, 2013.

\bibitem{paoli2020stability}
Gloria Paoli, Gianpaolo Piscitelli, and Rossano Sannipoli.
\newblock A stability result for the steklov laplacian eigenvalue problem with
  a spherical obstacle.
\newblock {\em arXiv preprint arXiv:2005.04449}, 2020.

\bibitem{schneider2014convex}
Rolf Schneider.
\newblock {\em Convex bodies: the Brunn--Minkowski theory}.
\newblock Number 151. Cambridge university press, 2014.

\bibitem{seo2021shape}
Dong-Hwi Seo.
\newblock A shape optimization problem for the first mixed steklov--dirichlet
  eigenvalue.
\newblock {\em Annals of Global Analysis and Geometry}, 59(3):345--365, 2021.

\bibitem{trudinger1967harnack}
Neil~S Trudinger.
\newblock On harnack type inequalities and their application to quasilinear
  elliptic equations.
\newblock {\em Communications on Pure and Applied Mathematics}, 20(4):721--747,
  1967.

\bibitem{verma2018bounds}
Sheela Verma.
\newblock Bounds for the steklov eigenvalues.
\newblock {\em Archiv der Mathematik}, 111(6):657--668, 2018.

\bibitem{verma2020eigenvalue}
Sheela Verma and G~Santhanam.
\newblock On eigenvalue problems related to the laplacian in a class of doubly
  connected domains.
\newblock {\em Monatshefte f{\"u}r Mathematik}, 193(4):879--899, 2020.

\end{thebibliography}

\end{document}